\documentclass[12pt,reqno]{amsart}

\usepackage{amsmath,amsthm,amssymb,comment,fullpage}
\usepackage{times}
\usepackage[T1]{fontenc}
\usepackage{mathrsfs}
\usepackage{latexsym}
\usepackage[dvips]{graphics}
\usepackage{epsfig}
\usepackage{accents}
\usepackage{float}
\usepackage{amsmath,amsfonts,amsthm,amssymb,amscd}
\input amssym.def
\input amssym.tex
\usepackage{color}
\usepackage{hyperref}
\usepackage{url}
\usepackage{breakurl}
\usepackage{comment}
\usepackage{mathtools}
\newcommand{\bburl}[1]{\textcolor{blue}{\url{#1}}}

\newcommand{\E}{\mathbb{E}}

\renewcommand{\E}{\mathbb{E}}

\numberwithin{equation}{section}

\theoremstyle{plain}
\newtheorem{thm}{Theorem}[section]

\newtheorem{prop}[thm]{Proposition}

\newtheorem{theorem}[thm]{Theorem}
\newtheorem{lemma}[thm]{Lemma}

\theoremstyle{definition}
\newtheorem{defi}[thm]{Definition}
\newtheorem{remark}[thm]{Remark}
\newtheorem{exa}[thm]{Example}

\newcommand\be{\begin{equation}}
\newcommand\ee{\end{equation}}
\newcommand\bea{\begin{eqnarray}}
\newcommand\eea{\end{eqnarray}}
\newcommand\bi{\begin{itemize}}
	\newcommand\ei{\end{itemize}}
\newcommand\ben{\begin{enumerate}}
	\newcommand\een{\end{enumerate}}
\newcommand\bc{\begin{center}}
	\newcommand\ec{\end{center}}
\newcommand\ba{\begin{array}}
	\newcommand\ea{\end{array}}

\newcommand{\R}{\ensuremath{\mathbb{R}}}
\newcommand{\C}{\ensuremath{\mathbb{C}}}
\newcommand{\Z}{\ensuremath{\mathbb{Z}}}

\newcommand{\N}{\mathbb{N}}

\newcommand{\hr}[1]{\href{#1}{\url{#1}}}

\renewcommand \l {\lambda}

\newcommand{\curleps}{\varepsilon}

\newcommand{\pfrac}[2]{\left(\frac{#1}{#2}\right)}

\newcommand{\tth}{^{\operatorname{th}}}

\newcommand{\Tr}{\text{Tr}}
\DeclareMathOperator{\tr}{Tr}
\newcommand{\etr}{\mathbb{E}\;\tr\;}

\newcommand*{\reff}[1]{\hyperref[#1]{\ref{#1}}}

\makeatletter
\let\@@pmod\pmod
\DeclareRobustCommand{\pmod}{\@ifstar\@pmods\@@pmod}
\def\@pmods#1{\mkern4mu({\operator@font mod}\mkern 6mu#1)}
\makeatother

\title{Spectral Distributions of Periodic Random Matrix Ensembles}

\author{Roger Van Peski}
\email{\textcolor{blue}{\href{mailto:rpeski@princeton.edu}{rpeski@princeton.edu}}}
\address{Department of Mathematics, Princeton University, Princeton, NJ 08544}

\thanks{The author would like to thank Steven J. Miller, Arup Bose, Nhi Truong and Peter Cohen for helpful discussions, thank the anonymous referees for helpful comments, and express gratitude to the Umashankar family for their hospitality while the main research was undertaken.}



\date{\today}

\begin{document}

\begin{abstract}
Kolo{\u{g}}lu, Kopp and Miller compute the limiting spectral distribution of a certain class of real random matrix ensembles, known as $k$-block circulant ensembles, and discover that it is exactly equal to the eigenvalue distribution of an $k \times k$ Gaussian unitary ensemble. We give a simpler proof that under very general conditions which subsume the cases studied by Kolo{\u{g}}lu-Kopp-Miller, real-symmetric ensembles with periodic diagonals always have limiting spectral distribution equal to the eigenvalue distribution of a finite Hermitian ensemble with Gaussian entries which is a `complex version' of a $k \times k$ submatrix of the ensemble. 
We also prove an essentially algebraic relation between certain periodic finite Hermitian ensembles with Gaussian entries, and the previous result may be seen as an asymptotic version of this for real-symmetric ensembles. The proofs show that this general correspondence between periodic random matrix ensembles and finite complex Hermitian ensembles is elementary and combinatorial in nature.
\end{abstract}

\maketitle

\section{Introduction}

Random matrix theory mainly studies the distribution of eigenvalues of matrices sampled from some \emph{ensemble} ($N \times N$ matrix-valued random variable) in the limit as the matrix size $N \to \infty$. Wigner  \cite{Wig1, Wig2, Wig3, Wig4, Wig5} initially used random matrices to model energy levels of heavy nuclei, but they have since found uses across many fields, mainly analytic number theory \cite{Mon, KS1, KS2, KeSn} and their original domain of physics. A very important classical role was played by the three \emph{classical random matrix ensembles}, the Gaussian Orthogonal, Unitary and Symplectic Ensembles respectively. Each has the important property that the probability measure is invariant under $A \mapsto CAC^*$ for any $C$ which is, respectively, orthogonal, unitary or symplectic. The probability measures have an explicit description, which we give for the GUE since it will be important later. 

\begin{defi}\label{def:gue}
The $N \times N$ \emph{Gaussian Unitary Ensemble (GUE)} is the random matrix ensemble given by 
\begin{equation}
\begin{pmatrix}
c_{0,0} & c_{0,1} & \cdots & c_{0,N-1} \\
c_{1,0} & c_{1,1} & \cdots & c_{1,N-1} \\
\vdots & \vdots & \ddots & \vdots \\
c_{N-1,0} & c_{N-1,1} & \cdots & c_{N-1,N-1} \\
\end{pmatrix},
\end{equation}
with the $c_{\ell,j}$ defined as follows. For $\ell \neq j$ $c_{\ell,j} \sim \mathcal{N}_\C(0,1)$ 
are iid subject only to the restriction $c_{j,\ell} = \overline{c_{\ell,j}}$ (i.e. the ensemble is Hermitian). Furthermore, $c_{j,j}$ are iid $\mathcal{N}_\R(0,1)$. Here $\mathcal{N}_\R(\mu,\sigma)$ denotes the real Gaussian distribution of mean $\mu$ and variance $\sigma$, and $\mathcal{N}_\C(0,1)$ is the distribution of $\frac{a+ib}{\sqrt{2}}$ for $a,b \sim \mathcal{N}_\R(0,1)$ iid.
\end{defi}

A central concept in random matrix theory is \emph{universality}, one flavor of which is the phenomenon that many different ensembles have semicircular limiting spectral distribution, in particular the GOE, GUE and GSE \cite{Meh}. For instance, the entries need not be Gaussian, but can indeed be any sufficiently nice random variables, and there is an ever-growing body of literature on how weak `sufficiently nice' can be made \cite{Tao, TV, TVK}. Classical ensembles such as the GUE are in some sense very `free,' in that (in the case above) the only dependent entries are those transpose entries needed to ensure that it is Hermitian. Besides studying these classical ensembles, a substantial bulk of the random matrix theory literature is devoted to the eigenvalue distributions of various special ``patterned'' ensembles with often non-semicircular limiting spectral distribution. These may consist of Toeplitz, Hankel or various other types of matrices, for which there are additional restrictions on the entries \cite{Bai, BasBo1, BasBo2, BanBo, BLMST, BCG, BHS1, BHS2, BM, BDJ, GKMN, HM, JMRR, JMP, Kar, KKM, LW, MMS, MNS, MSTW, McK, Me, Sch}. 

What is very interesting is when the theory of such patterned ensembles relates back to the classical ensembles. In \cite{KKM}, Kolo{\u{g}}lu, Kopp and Miller study ensembles of block circulant matrices.

\begin{defi}\label{def:circulant}
For $k|N$, an $N \times N$ \textbf{$k$-block circulant ensemble} is one where the entries $a_{i,j}$ satisfy
\begin{align}
a_{i,j} &= a_{i+k \pmod*{N},j+k \pmod*{N}}\\
a_{i,j} &= a_{j,i},
\end{align} 
where the random variables $a_{i,j}$ are iid (other than the above restrictions) real random variables with mean $0$, variance $1$ and finite higher moments. Equivalently, this is the ensemble of block matrices
\begin{equation}
\begin{pmatrix}
B_0 & B_1 & \cdots & B_{N/k-1} \\
B_{-1} & B_0 & \cdots & B_{N/k-2} \\
B_{-2} & B_{-1} & \cdots & B_{N/k-3} \\
\vdots & \vdots & \ddots & \vdots \\
B_{1-N/k} & B_{2-N/k} & \cdots & B_{0}
\end{pmatrix},
\end{equation}
where each block $B_i$ consists of iid mean-$0$ variance-$1$ random variables subject only to the restrictions that the matrix is symmetric and $B_{N/k-i} = B_{-i}$ for all $i$.
\end{defi}

\begin{remark}\label{rmk:comments_on_circulant}
\begin{itemize}
    \item It will turn out that the actual distribution of the entries has no effect on the limiting spectral distribution provided that it satisfies the conditions in Definition \ref{def:circulant}.
    \item Without the condition that the indices of the RHS in $a_{i,j} = a_{i+k \pmod*{N},j+k \pmod*{N}}$ are taken modulo $N$, this definition describes a Toeplitz ensemble, which have also been studied extensively.
    \item When $k=N$, the $k$-block circulant ensemble is simply an iid symmetric ensemble, and when $k=1$ it is a standard circulant ensemble.
\end{itemize}
\end{remark}

There are several classical results on circulant matrices which are not difficult to prove, and which we recall here for context--all and many more may be found in \cite{davis}. 

\begin{prop}
The algebra of $N \times N$ real (resp. complex) circulant matrices (not necessarily Hermitian) is isomorphic to the group algebra $\R[\Z/N\Z]$ (resp. $\C[\Z/N\Z]$).
\end{prop}

This is easy to verify, as a circulant matrix is a linear combination of powers of the permutation matrix corresponding to the shift $(12\cdots N)$, and such a shift generates a copy of the cyclic group of order $N$.

\begin{prop}\label{prop:circ_eigs}
Let $A = (a_{ij})_{0 \leq i,j \leq N-1}$ be an $N \times N$ circulant matrix ($1$-block circulant by the above definition). Then the eigenvectors of $A$ are given by 
\begin{equation}
    \begin{pmatrix}
    1 \\ \zeta \\ \zeta^2 \\ \vdots \\ \zeta^{N-1}
    \end{pmatrix}
\end{equation}
for $\zeta$ an $N\tth$ root of unity (not necessarily primitive), and the corresponding eigenvalues are $\lambda_\zeta = \sum_{i=0}^{N-1} a_{0i}\zeta^i$.
\end{prop}

This too is easy to verify by inspection. We note that if a circulant matrix is Hermitian, the above formula for its eigenvalues gives a real number as expected, since $a_{0,i} = \overline{a_{0,-i}}$. As noted in Remark \ref{rmk:comments_on_circulant}, circulant matrices are a special case of the $k$-block circulant ensemble, as are iid symmetric ensembles. The $k$-block circulant ensemble thus interpolates between the circulant case and the iid symmetric case. In the former Proposition \ref{prop:circ_eigs} gives an explicit formula for the eigenvalues, which converges to a Gaussian as $N \to \infty$ by the central limit theorem. In the latter, classical results ensure that in the large $N$ limit the eigenvalues have semicircular distribution. Hence without any further calculation, one should expect the limiting spectral distribution of the $k$-block circulant ensembles to interpolate between the Gaussian ($k=1$) and the semicircle ($k \to \infty$). 

\cite{KKM} explicitly finds the limiting eigenvalue distribution of the $k$-block circulant ensemble as $N \to \infty$, showing that it in fact converges to that of the \emph{finite} $k \times k$ GUE. This is Gaussian for $k=1$ and converges to a semicircle as $k \to \infty$, confirming the heuristic argument of the previous paragraph. Another example of limiting spectral distributions converging to the eigenvalue distribution of a finite ensemble is provided by \cite{BC+}, which considers the fluctuations about large trivial eigenvalues of certain periodic ensembles spiked with constant entries. In this case, the distribution converges to a modified GOE with diagonal entries all $0$. The combinatorial framework used by the author and others in \cite{BC+} inspired the proof of the main result of this paper. 

Just as with the classical theory, the convergence of the spectral distribution of the $k$-block circulant ensemble to the $k \times k$ GUE begs the question of what happens when one adds restrictions on entries of the original ensemble. The $k$-block circulant ensembles are in some sense the `most free' symmetric ensemble for which entries are periodic down the diagonals with period $k$. Because these diagonals `wrap' around the matrix and have two connected components except in the case of the main diagonal, we use the following terminology in the rest of the paper.

\begin{defi}\label{def:wrapped_diagonal}
Given an $N \times N$ ensemble $M = (a_{i,j})_{0 \leq i,j \leq N-1}$, we refer to the set of entries $\{a_{i,j}: j-i \equiv t \pmod{N}\}$ as the \textbf{wrapped diagonal} corresponding to $t$.
\end{defi}

In the $4$-block circulant case, for instance, the random variables on the wrapped diagonal corresponding to $t$ are of the form $a_{0,t},a_{1,t+1},a_{2,t+2},a_{3,t+3},a_{0,t},a_{1,t+1},\ldots$. In this case, we say that the ensemble is specified by the \emph{diagonal pattern} $(a,b,c,d)$, i.e. each diagonal consists of four independent random variables $a,b,c,d,a,b,c,d,\ldots$ repeating with period $4$. One could consider alternate patterns, e.g. $(a,b,a,b)$, which gives a $2$-block circulant ensemble, or some stranger asymmetric pattern such as $(a,b,a,c)$. \cite[Appendix A]{KKM} asks the natural question of whether the limiting spectral distribution of such a patterned periodic ensemble depends solely on the frequency of different random variables on the wrapped diagonal, or whether the order in which they appear matters as well, for example whether the ensembles given by $(a,b,a,b)$ and $(a,b,b,a)$ have the same limiting spectral distribution. They show by a moment bounding argument that this is not the case, and factors other than the frequency with which random variables appear on the wrapped diagonal provably affect the result. However, they do not find a closed form for the moments of the limiting spectral distribution. 

In this paper, we show in Theorem \ref{thm:main} that the convergence of the spectral distribution of a periodic random matrix ensemble to that of a finite random matrix ensemble, observed by \cite{KKM} in the circulant/GUE correspondence, is in fact a general combinatorial feature of any periodic random matrix ensemble. The proof is by the method of moments and an argument which recovers the moments of the real and complex Gaussians in the $k \times k$ ensemble from the combinatorics of choosing entries in the original ensemble. It yields as corollaries a shorter and more elementary proof of the main result of \cite{KKM}, and an explicit form for the limiting spectral distributions of the patterned circulant ensembles for which \cite[Appendix A]{KKM} obtains only bounds and numerics. However, it should be noted that the ensembles treated by Theorem \ref{thm:main} are more general in that different diagonals are not required to have the same pattern of entries as in \cite[Appendix A]{KKM}.

\section{Preliminaries and main theorem}

Throughout the paper, we will zero-index matrix entries, e.g. the top-left entry is denoted $a_{0,0}$. $N$ will always denote the size of an ensemble, and $k$ and $M_N$ will be as used below. All ensembles treated in this paper, unless otherwise specified, will be Hermitian, so the eigenvalues will be assumed to be real without further comment. The law of such a random matrix ensemble $M_N$ is a probability measure on $M_{N \times N}(\R)$ given by the product measure coming from the distributions of the individual random variable entries, and any function such as $\Tr \; M_N$ is then an $\R$-valued random variable on this space.

The results may be stated most succinctly using a modification of the language of link functions of \cite{BS}, which is just a convenient way to specify entries in a random matrix ensemble to be either equal or independent.

\begin{defi}\label{def:link_function}
A \textbf{$k$-link function}, or link function when $k$ is clear, is any function $f:(\Z/k\Z)^2 \to S$ for some index set $S$, which for us will always be $\Z/k\Z$. We often treat $f$ as a function on $\Z^2$, where we implicitly precompose with the quotient map $\Z^2 \to  (\Z/k\Z)^2$.

Given $f$ a $k$-link function, an $N \times N$ real random matrix ensemble is \textbf{$f$-linked} if its entries are real random variables with mean $0$, variance $1$ and finite higher moments, which are iid apart from the restriction that $a_{i,j} = a_{i',j'}$ if one of the following is true:
\begin{enumerate}
    \item $i=j',j=i'$ (symmetry).
    \item $i-j \equiv i'-j' \pmod{N}$ (both entries lie on the same wrapped diagonal) and $$f(i , j ) = f(i', j').$$ 
\end{enumerate}
If $X$ is the distribution of one of the entries of an $f$-linked ensemble, we denote the ensemble by $M_{f,N,X}$. 
\end{defi}

It is very important to note that the symmetry restriction $a_{i,j}=a_{j,i}$ is \emph{in addition to} the restriction imposed by the link function $f$, and in general the link function itself will not impose this. It is always possible to write one which will, but this generally results in a more complicated function for which the restrictions it imposes are less easy to see by inspection, which is why the definition is given as above.

\begin{remark}
It should be noted that real ensembles defined by link functions never have dependent entries which are not equal. It does not appear that the method of proof used lends itself to such ensembles, but it would be interesting to see whether a version of Theorem \ref{thm:main} holds for them.
\end{remark}

\begin{exa}
The link function of an $N \times N$ $k$-block circulant ensemble may be given as 
$$f(i ,j ) = i \pmod{k}.$$
This ensures that along each wrapped diagonal, entries repeat with period $k$.
\end{exa}


\begin{exa}\label{ex:circulant}
The following are examples of $2$-periodic ensembles, the first being simply a circulant ensemble. Here the $c_i$ and $d_i$ are iid random variables with mean $0$, variance $1$, and finite higher moments.
\begin{equation*}
\left(\begin{array}{cc|cc|cc}
c_0 & c_1 & c_2 & c_3 & c_2 & d_1\\
c_1 & d_0 & d_1 & d_2 & c_3 & d_2\\ \cline{1-6}
c_2 & d_1 & c_0 & c_1 & c_2 & c_3\\
c_3 & d_2 & c_1 & d_0 & d_1 & d_2\\ \cline{1-6}
c_2 & c_3 & c_2 & d_1 & c_0 & c_1\\
d_1 & d_2 & c_3 & d_2 & c_1 & d_0
\end{array}\right),\ \ \ \ \ \ \
\left(\begin{array}{cc|cc|cc}
c_0 & c_1 & c_2 & c_3 & c_2 & d_1\\
c_1 & c_0 & d_1 & c_2 & c_3 & c_2\\ \cline{1-6}
c_2 & d_1 & c_0 & c_1 & c_2 & c_3\\
c_3 & c_2 & c_1 & c_0 & d_1 & c_2\\ \cline{1-6}
c_2 & c_3 & c_2 & d_1 & c_0 & c_1\\
d_1 & c_2 & c_3 & c_2 & c_1 & c_0
\end{array}\right),\ \ \ \ \ \ \
\left(\begin{array}{cc|cc|cc}
c_0 & c_1 & c_2 & c_3 & c_2 & c_1\\
c_1 & d_0 & c_1 & d_2 & c_3 & d_2\\ \cline{1-6}
c_2 & c_1 & c_0 & c_1 & c_2 & c_3\\
c_3 & d_2 & c_1 & d_0 & c_1 & d_2\\ \cline{1-6}
c_2 & c_3 & c_2 & c_1 & c_0 & c_1\\
c_1 & d_2 & c_3 & d_2 & c_1 & d_0
\end{array}\right).
\end{equation*}
Their link functions are
\begin{align}
    f^{(1)}(i,j) &= i \pmod{2}, \\
    f^{(2)}(i,j) &= (i-j) \cdot i \pmod{2}, \\
    f^{(3)}(i,j) &= (i-j+1) \cdot i \pmod{2},
\end{align}
respectively.
\end{exa}

For the $k$-block circulant ensemble, we gave an alternate definition in Definition \ref{def:circulant} in terms of the $k \times k$ blocks $B_i$. In the $k$-block circulant case, two blocks $B_i,B_j$ were independent unless they lie on the same wrapped diagonal or its transpose. However, for general $k$-periodic ensembles this is not true; for example, in the third ensemble of Example \ref{ex:circulant}, we have blocks
\begin{equation}
    B_0 = \begin{pmatrix} c_0 & c_1 \\ c_1 & d_0 \end{pmatrix} \text{ and }B_1 = \begin{pmatrix} c_2 & c_3 \\ c_1 & d_2 \end{pmatrix},
\end{equation}
which share the common entry $c_1$ and hence are not independent, but the blocks $B_0$ and $B_1$ lie on wrapped diagonals which are not transposes of one another. It is better to think not of the blocks $B_i$ and the diagonals they lie on, but simply the wrapped diagonals of the whole matrix given by $\{a_{i,j}: i-j \equiv r \pmod{N}\}$ for each $r$. Because of the $i-j \pmod{N}$ condition from Definition \ref{def:link_function}, two entries on distinct wrapped diagonals which are not transposes of one another will always be independent. This is a crucial fact for the proofs of this section and will be used repeatedly.

Because the aim of this paper is to treat the $N \to \infty$ limit of ensembles with the same structure given by a link function, we state what this means concretely below.

\begin{defi}\label{def:k_periodic_fam}
Fix a real-valued random variable $X$ with mean $0$, variance $1$, and finite higher moments, and a $k$-link function $f$. A \textbf{$k$-periodic random matrix ensemble family}, or simply $k$-periodic family, is a sequence of random matrix ensembles $(M_{f,N,X})_{N \in k\N}$ where $M_{f,N,X}$ is as in Definition \ref{def:link_function}.
\end{defi}

We now define the `complex version' of a $k$-periodic family alluded to in the Introduction. 

\begin{defi}\label{def:complexification}
Given positive integers $k|N$ and a $k$-link function $f$, the \textbf{$N \times N$ complex companion ensemble} $\tilde{M}_{f,N}$ is the $N \times N$ Hermitian ensemble with entries $c_{\ell,j}$ distributed as
\begin{equation}
    c_{\ell,j} \sim \begin{cases}
    \mathcal{N}_\R(0,1) & \text{ if }2(j-\ell) \equiv 0 \pmod{N} \text{ and }f(j,\ell)=f(\ell,j) \\
    \mathcal{N}_\C(0,1) & \text{ otherwise}
    \end{cases}.
\end{equation}
which are independent apart from the following restrictions:
\begin{enumerate}
    \item $c_{\ell,j} = \overline{c_{j,\ell}}$  (i.e. the ensemble is Hermitian). 
    \item $c_{\ell,j} = c_{m,n}$ if $\ell-j \equiv m-n \pmod{N}$ and $f(\ell,j) = f(m,n)$.
\end{enumerate}
\end{defi}

\begin{remark}\label{rmk:split_diagonal}
The condition $2(j-\ell) \equiv 0 \pmod{N}$ is worth explaining. If $N$ is odd, it is equivalent to $j\equiv \ell \pmod{N}$, and the entries $c_{\ell,j}$ satisfying this are the ones on the main diagonal; the extra condition $f(j,\ell)=f(\ell,j)$ is trivially satisfied in this case. It is obviously necessary that the diagonal entries be real rather than complex for the matrix to be Hermitian.

If $N$ is even, there are additional possibilities. Visually, if ones breaks the matrix into $4$ equal-sized blocks, the matrix entries $c_{\ell,j}$ with $2(j-\ell) \equiv 0 \pmod{N}$ and $j-\ell \not \equiv 0 \pmod{N}$ are the ones on the diagonals of the upper-right and lower-left blocks, depending whether $j-\ell$ is $N/2$ or $-N/2$. Together, these two diagonals make up a single wrapped diagonal of our ensemble, which we refer to as the \textbf{split diagonal}. It and the main diagonal are the only two wrapped diagonals which are their own transpose. 

In general there are two ways that entries of a $k$-periodic ensemble or its complex companion ensemble may be dependent: either they are transposes of one another and hence complex conjugates, or else the link function constrains them to be equal. The entries on the main diagonal and the split diagonal are the only ones for which both options are available and we can have $c_{\ell,j} = \overline{c_{j,\ell}}$ (from the Hermitian property) $c_{\ell,j} = c_{j,\ell }$ (from the link function). For the main diagonal entries with $j=\ell$, this always occurs, and they must thus be real. For the split diagonal entries, this may occur for certain link functions and result in real entries, as shown in the example below. We do not have real entries on the split diagonal for every choice of link function; for instance, for the one $f(i,j) = i \pmod{k}$ corresponding to the $k$-block circulant ensemble, the complex companion ensemble is the GUE, and its only real entries are on the main diagonal.
\end{remark}

\begin{exa}
Let $f^{(i)}, i = 1,2,3$ be as in Example \ref{ex:circulant}. Then the complex companion ensemble of the $2$-circulant link function $f^{(1)}$ is
\begin{equation}
\tilde{M}_{f^{(1)},2} = 
    \begin{pmatrix}
    a & b \\
    \bar{b} & c
    \end{pmatrix}
\end{equation}
where $a,c \sim \mathcal{N}_\R(0,1)$ iid and $b \sim \mathcal{N}_\C(0,1)$; this is just the $2 \times 2$ GUE. The complex companion ensemble of $f^{(2)}$ is 
\begin{equation}
\tilde{M}_{f^{(2)},2} = 
    \begin{pmatrix}
    a & b \\
    \bar{b} & a
    \end{pmatrix}
\end{equation}
with $a \sim \mathcal{N}_\R(0,1)$ and $b \sim \mathcal{N}_\C(0,1)$ are independent. The complex companion ensemble of $f^{(3)}$ is 
\begin{equation}
\tilde{M}_{f^{(3)},2} = 
    \begin{pmatrix}
    a & b \\
    b & c
    \end{pmatrix}
\end{equation}
with $a,b,c \sim \mathcal{N}_\R(0,1)$ iid. Note that $b$ is now real, as per Remark \ref{rmk:split_diagonal}.
\end{exa}

All that remains is a few standard definitions.

\begin{defi}\label{def:esd}
The \textbf{empirical spectral measure} of a (fixed, nonrandom) $N\times N$ matrix $A$ is the probability measure
\begin{equation}\label{eq:esd}
\nu_A = \frac{1}{N}\sum_{i = 1}^{N} \delta\left(x - \frac{\lambda_{i}}{\sqrt{N}}\right),
\end{equation}
where the $\{\lambda_{i}\}_{i = 1}^{N}$ are the eigenvalues of $A$ and $\delta$ is the Dirac measure. Its $m \tth$ moment is denoted by 
\begin{equation}
\nu_A^{(m)} := \int_{-\infty}^\infty x^m d\nu_A.
\end{equation}
\end{defi}

Let $M_N$ be an arbitrary $N \times N$ Hermitian random matrix ensemble. We will view the moments $\nu_{M_N}^{(m)}$ of the empirical spectral measure of a randomly chosen matrix of the ensemble $M_N$ as real-valued random variables. The empirical spectral measure also defines a discrete random variable which takes values $\frac{1}{\sqrt{N}}\lambda_1,\ldots,\frac{1}{\sqrt{N}}\lambda_N$ with equal probability; therefore, given an ensemble $M_N$, we have a $\R$-valued random variable on the probability space $M_{N \times N}(\R) \times \{1,\ldots,N\}$ corresponding to choosing a matrix from $M_{N \times N}(\R)$ and then randomly choosing one of its eigenvalues and normalizing by $\frac{1}{\sqrt{N}}$. This random variable defines a measure on $\R$ which we refer to as the \textbf{eigenvalue distribution} of the ensemble $M_N$ and denote it by $\mu_{M_N}$. 

It is worth noting that there is a loss of information implicit in the definition of the eigenvalue distribution, e.g. spacings and correlations of eigenvalues; it would be interesting to see whether Theorem \ref{thm:main} can be strengthened to say something about such statistics as well, but such analysis is beyond the scope of the method of proof used. 

Theorem \ref{thm:main_complex} shows that the eigenvalue distribution of an $N \times N$ $k$-periodic complex Hermitian ensemble is the same as that of the $k \times k$ complex Hermitian ensemble with the same link function. It is in some sense a purely algebraic statement relating two complex ensembles, with no asymptotics necessary. Theorem \ref{thm:main} shows that the limiting spectral distribution of a $k$-periodic family of real ensembles is determined by the eigenvalue distribution of the $k \times k$ complex Hermitian ensemble with the same link function, but in contrast to Theorem \ref{thm:main_complex} this is in an asymptotic sense. It is interesting that both reduce to the same $k \times k$ ensemble, but in the real case the $N \to \infty$ limit is required to smooth out the real ensemble to the complex one. As might be expected, the proof of Theorem \ref{thm:main} is more involved, but the proof of Theorem \ref{thm:main_complex} provides an instructive summary of some of the techniques used.

\begin{theorem}\label{thm:main_complex}
Let $k|N$ be positive integers and $f$ a $k$-link function. Then the complex companion ensembles $\tilde{M}_{f,k}$ and $\tilde{M}_{f,N}$ have the same eigenvalue distribution. 
\end{theorem}

\begin{theorem} \label{thm:main}
Let $f$ be a $k$-link function and $(M_{f,N,X})_{N \in k\N}$ be a $k$-periodic family, and $(A_N)_{N \in k\N}$ a sequence with $A_N$ chosen from the ensemble $M_{f,N,X}$. Then with probability $1$ in the product measure, the sequence of empirical spectral measures $(\nu_{A_N})_{N \in k\N}$ converge weakly to the eigenvalue distribution $\mu_{\tilde{M}_{f,k}}$ of the complex companion ensemble $\tilde{M}_{f,k}$.
\end{theorem}

\begin{remark}
Theorem \ref{thm:main} would not make sense without Theorem \ref{thm:main_complex}, and indeed the former follows as a corollary to the latter. If $f$ is a $k$-link function $(M_{f,N,X})_{N \in k\N}$ is a $k$-periodic family, then $(M_{f,N,X})_{N \in tk\N}$ is also a $tk$-periodic family for any $t \in \N$. Hence Theorem \ref{thm:main} shows that the limiting eigenvalue distributions of $(M_{f,N,X})_{N \in k\N}$ and $(M_{f,N,X})_{N \in tk\N}$ are equal to those of the complex companion ensembles $\tilde{M}_{f,k}$ and $\tilde{M}_{f,tk}$ respectively. It is not hard to show directly that $(M_{f,N,X})_{N \in k\N}$ and $(M_{f,N,X})_{N \in tk\N}$ must have the same limiting eigenvalue distribution, in the sense of Theorem \ref{thm:main}, if such a distribution exists. Hence $\tilde{M}_{f,k}$ and $\tilde{M}_{f,tk}$ have the same eigenvalue distribution.
\end{remark}

Of course, Theorem \ref{thm:main_complex} is rightly viewed as a simpler algebraic skeleton underlying Theorem \ref{thm:main}, and thus proving it as a corollary to the latter is needlessly roundabout; it may also be proved directly using a shorter version of the combinatorial arguments used for Theorem \ref{thm:main} in the next section. Since Theorem \ref{thm:main} is our main focus and we wish to avoid too much repetition, in this section we simply sketch how the proof of Theorem \ref{thm:main_complex} reduces to the same combinatorial problem, and show all the details in the proof of Theorem \ref{thm:main} in the next section.

Let us sketch the argument. The empirical spectral measures $\nu_{M_{f,k,X}}$ and $\nu_{M_{f,N,X}}$ are random measures, hence their $m\tth$ moments $\nu_{M_{f,k,X}}^{(m)},\nu_{M_{f,N,X}}^{(m)}$ are just real-valued random variables. We claim it suffices to show $\E[\nu_{M_{f,k,X}}^{(m)}] = \E[\nu_{M_{f,N,X}}^{(m)}]$ for all $m \in \N$. Define the random variable $X_k$ (resp. $X_N$) to be $\frac{1}{\sqrt{k}}$ (resp. $\frac{1}{\sqrt{N}}$) times a uniformly random eigenvalue of a random matrix drawn from $M_{f,k,X}$ (resp. $M_{f,N,X}$), i.e. a random variable whose distribution is the eigenvalue distribution of $M_{f,k,X}$ (resp. $M_{f,N,X}$). Then its $m\tth$ moment is
\begin{equation}
    \E[X_k^m] = \frac{1}{k} \E\left[ \sum_{1 \leq i \leq k} \pfrac{\l_i(M_{f,k,X})}{\sqrt{k}}^m\right] = \E[\nu_{M_{f,k,X}}^{(m)}].
\end{equation}
Hence if $\nu_{M_{f,k,X}}^{(m)}=\nu_{M_{f,N,X}}^{(m)}$ in expectation, then 
\begin{equation}
    \E[X_k^m] = \E[\nu_{M_{f,k,X}}^{(m)}] = \E[\nu_{M_{f,N,X}}^{(m)}] = \E[X_N^m].
\end{equation}
Since the moments are equal, it remains to show that they uniquely characterize the measure. In the proof of Theorem \ref{thm:main} we will show by Carleman's condition (Proposition \ref{prop:carleman}) that for any $k \times k$ complex Hermitian ensemble given by a link function, this is indeed the case. 

We now show $\E[\nu_{M_{f,k,X}}^{(m)}] = \E[\nu_{M_{f,N,X}}^{(m)}]$. The eigenvalue-trace formula and linearity of expectation give that 
\begin{equation}\label{eq:eigtrace}
\E\left[\nu^{(m)}_{M_{f,N,X}}\right] =  \frac{1}{N^{m/2+1}}\sum_{1\le i_1,\ldots,i_m\le N} \E \left[a_{i_1,i_2}\cdots a_{i_{m-1},i_{m}}a_{i_m, i_1}\right]
\end{equation}
where the $a_{i_j,i_{j+1}}$ are the entries of $M_{f,N,X}$. The products $a_{i_1,i_2}\cdots a_{i_{m-1},i_{m}}a_{i_m, i_1}$ will be referred to as \textbf{cyclic products}, as their indices are cyclic, and we will often write $a_{i_j,i_{j+1}}$ for brevity even though in the case $j=m$ we are actually taking $j+1 \pmod{m}$. It suffices to show 
\begin{equation}
    \frac{1}{N^{m/2+1}}\sum_{1\le i_1,\ldots,i_m\le N} \E \left[a_{i_1,i_2}\cdots a_{i_{m-1},i_{m}}a_{i_m, i_1}\right] = \frac{1}{k^{m/2+1}}\sum_{1\le i_1,\ldots,i_m\le k} \E \left[b_{i_1,i_2}\cdots b_{i_{m-1},i_{m}}b_{i_m, i_1}\right]
\end{equation}
where the $b_{i_j,i_{j+1}}$ are the entries of $M_{f,k,X}$. This reduces to setting up a correspondence between the cyclic products $b_{i_1,i_2}\cdots b_{i_{m-1},i_{m}}b_{i_m, i_1}$ and related sets of cyclic products $a_{i_1,i_2}\cdots a_{i_{m-1},i_{m}}a_{i_m, i_1}$ which takes advantage of the combinatorics of Gaussian moments, and is discussed in full detail with the formalism of `patterns' and `matchings' in the next section. This concludes the sketch.

\section{Proof of Theorem \ref{thm:main}}

This section will be entirely devoted to the proof of Theorem \ref{thm:main}. We utilize the standard moment convergence argument. This involves first showing that for every $m$, the moments $\nu_{M_{f,N,X},N}^{(m)}$ converge in expectation to the moments of $\mu_{\tilde{M}_{f,k}}$, then showing that the random variables $\nu_{M_{f,N,X},N}^{(m)}$ converge to their expectations in the limit. In this proof, $A_N$ will be used to denote an element of $M_{N \times N}(\R)$, usually chosen randomly with respect to the law of $M_{f,N,X}$.

Recall \eqref{eq:eigtrace}. We note that $\E \left[a_{i_1,i_2}\cdots a_{i_{m-1},i_{m}}a_{i_m, i_1}\right] = 0$ if any of the $a_{i,j}$ appearing in the cyclic product is independent of all others. Since $M_{f,N,X}$ is a $k$-periodic family, Definition \ref{def:link_function} implies that $a_{i,j}$ and $a_{i',j'}$ are independent if they do not lie on the same wrapped diagonal or its transpose. Hence most cyclic products have expectation $0$, and those which do not are heavily constrained. The crux of the proof lies in making the choices of the indices of cyclic products in the right order so as to mimic the combinatorics of a finite ensemble; the following discussion draws heavily on \cite{BC+}. Given a cyclic product $a_{i_1,i_2}\cdots a_{i_{m-1},i_{m}}a_{i_m, i_1}$, we first choose the congruence class of each index mod $k$. We formalize this with the notion of a pattern below. 

\begin{defi}\label{def:pattern}
A \textbf{$m$-pattern}, or simply pattern when $m$ is clear, is an element of $(r_1,\ldots,r_m) \in (\Z/k\Z)^m$. If a cyclic product $a_{i_1,i_2}\cdots a_{i_{m-1},i_{m}}a_{i_m, i_1}$ has $i_\ell \equiv r_\ell \pmod{k}$ for all $1 \leq \ell \leq m$, we say that it \textbf{conforms to} the pattern $(r_1,\ldots,r_m)$.
\end{defi}

\begin{exa}\label{ex:pattern}
Let $k=2$ and $m=4$. The cyclic products $a_{1,2}a_{2,1}a_{1,2}a_{2,1}$ and $a_{1,2}a_{2,1}a_{1,4}a_{4,1}$ both conform to the pattern $(1,2,1,2)$. However, if the entries $a_{i,j}$ are drawn from a $2$-block circulant ensemble for $N \geq 4$, then $\E[a_{1,2}a_{2,1}a_{1,2}a_{2,1}]$ is the fourth moment of $a_{1,2}$, but $\E[a_{1,2}a_{2,1}a_{1,4}a_{4,1}] = \E[a_{1,2}a_{2,1}] \cdot \E[a_{1,4}a_{4,1}] = 1$. 
\end{exa}

This shows that cyclic products conforming to the same pattern may have different expectations, depending on which of the indices $i_j$ coming from the same congruence class are \emph{actually equal}. Hence we would like, given a pattern, to specify extra constraints which ensure the resulting cyclic products have a certain expectation. This motivates Definition \ref{def:matching}.

\begin{defi}\label{def:matching}
Fix positive integers $m,k,N$ with $k|N$ and a $k$-link function $f$. A \textbf{matching} is an equivalence relation $\sim$ on $\{1,\ldots,m\}$, and a cyclic product $a_{i_1,i_2}\cdots a_{i_{m-1},i_{m}}a_{i_m, i_1}$ of entries from an ensemble $M_{f,N,X}$ \textbf{conforms to} $\sim$ if
\begin{itemize}
    \item $t \sim \ell$ if and only if either $i_{\ell+1}-i_\ell \equiv i_{t+1}-i_t \pmod{N}$ and  $f(i_\ell,i_{\ell+1}) = f(i_t,i_{t+1})$, or $\{i_\ell,i_{\ell+1}\}=\{i_t,i_{t+1}\}$. In other words, $t \sim \ell$ if and only if $a_{i_\ell,i_{\ell+1}} = a_{i_t,i_{t+1}}$.
\end{itemize} 
\end{defi}

\begin{exa}\label{ex:matching}
Let $k=2$ and $m=4$ as before. We have that the first cyclic product shown in Example \ref{ex:pattern} conforms to the trivial matching $\sim$ given by $1 \sim 2 \sim 3 \sim 4$, and the second cyclic product conforms to the matching $\sim'$ given by $1 \sim' 2, 3 \sim' 4$. 
\end{exa}

\begin{exa}
It is an important point that while matchings are defined independent of the ensemble, the condition of a cyclic product conforming to a matching is highly dependent on the specific ensemble's link function $f$, rather than just the indices $i_1,\ldots,i_m$. The condition of conforming to a pattern, however, is only dependent on the indices. For example, the cyclic product $a_{1,1}a_{1,2}a_{2,2}a_{2,1}$ conforms to the matching $2 \sim 4$ if the $a_{i,j}$ are taken from the $2$-block circulant ensemble (see Example \ref{ex:circulant}). However, it conforms to the matching $1 \sim 3, 2 \sim 4$ if the $a_{i,j}$ are from second ensemble shown in Example \ref{ex:circulant}.
\end{exa}

\begin{remark}\label{rmk:unique}
It is clear that for a fixed ensemble $M_{f,N,X}$, each cyclic product conforms to a unique matching, given by finding which of its entries are constrained to be equal, and conforms to a unique pattern, given by taking the indices of its entries mod $k$. However, different matchings may have different numbers of cyclic products conforming to them, or none at all. It is also clear that for fixed $m$, the number of patterns and matchings is finite.
\end{remark}

This suggests the following natural procedure to sum over all cyclic products: first sum over all patterns and matchings, then sum over every cyclic product conforming to them. In other words, for $A_N$ drawn from a $k$-periodic ensemble $M_{f,N,X}$ we have
\begin{equation}\label{eq:sum_order}
\etr M_{f,N,X}^m = \sum_{\substack{\text{patterns} \\ P}} \sum_{\substack{\text{matchings} \\ \sim}} \sum_{\substack{\text{Cyclic products $\Pi$ of length $m$} \\ \text{conforming to $P$ and $\sim$}}} \E[\Pi]
\end{equation}

The next lemma concerns for which patterns and matchings $\sum_{\substack{\text{Cyclic products $\Pi$ of length $m$} \\ \text{conforming to $P$ and $\sim$}}} \E[\Pi]$ is nonzero the limit $N \to \infty$.

\begin{lemma}\label{lem:match_pairs}
A pattern $P=(r_1,\ldots,r_m)$ and matching $\sim$ only contribute to $\E\left[\nu^{(m)}_{M_{f,N,X}}\right]$ in the limit if and only if
\begin{enumerate}
\item The equivalence classes of $\sim$ all have size $2$ (in particular $m$ must be even), and 
\item If $\ell \sim t$, then $r_\ell - r_{\ell+1} = -(r_t - r_{t+1})$ (in $\Z/k\Z$).
\end{enumerate}
In this case, 
\begin{equation}\label{eq:compute_contrib}
    \lim_{N \to \infty} \sum_{\substack{\text{Cyclic products $\Pi$ of length $m$ from $M_{f,N,X}$} \\ \text{conforming to $P$ and $\sim$}}} \E[\Pi] = (1/k)^{1+m/2},
\end{equation}
where as before the limit is taken along multiples of $k$. 
\end{lemma}
\begin{proof}
By \eqref{eq:eigtrace} we have that 
\begin{equation*}
\E\left[\nu^{(m)}_{M_{f,N,X}}\right] = \frac{1}{N^{m/2+1}}\sum_{1\le i_1,\ldots,i_m\le N} \E \left[a_{i_1,i_2}\cdots a_{i_{m-1},i_{m}}a_{i_m, i_1}\right].
\end{equation*}
We first show (1) is necessary, then show (2) is necessary, then show \eqref{eq:compute_contrib} assuming (1) and (2).

First, we note that if $\sim$ has an equivalence class $\{j\}$ of size $1$, then $\E \left[a_{i_1,i_2}\cdots a_{i_{m-1},i_{m}}a_{i_m, i_1}\right] = 0$ for any cyclic product conforming to $\sim$, since the entry $a_{i_j,i_{j+1}}$ is independent of the others and has mean $0$. Hence we may assume that each equivalence class has size $\geq 2$, so there are $\geq m/2$ equivalence classes.

To specify a cyclic product $a_{i_1,i_2}\cdots a_{i_{m-1},i_{m}}a_{i_m, i_1}$ conforming to $P$ and $\sim$, we first specify the differences $i_{j+1}-i_j$ mod $N$ for each $j$ (equivalently, which wrapped diagonal the entry $a_{i_j,i_{j+1}}$ is on), then finally specify one of the indices. Once all differences are specified mod $N$, choosing any index determines the whole cyclic product. Once the difference $i_{j+1}-i_j \pmod{N}$ is chosen for one element of an equivalence class under a matching $\sim$, it is specified (up to sign) for the other elements of the equivalence class. There are at most $2^m$ choices of sign. The difference $i_{1}-i_2$ is already specified mod $k$ by the pattern $P$, so there are $N/k$ choices. Once it is chosen, we must choose the next one; however, there may not be exactly $N/k$ choices, since we do not want to have entries $a_{i_j,i_{j+1}}$ and $a_{i_{j'},i_{j'+1}}$ in our cyclic product with $f(i_j,i_{j+1}) = f(i_{j'},i_{j'+1})$ and $i_j-i_{j+1} \equiv i_{j'}-i_{j'+1} \pmod{N}$ if $j \not \sim j'$. At most this precludes one choice for each difference already specified, so the number of choices for all differences is a product of factors $N$ minus a constant, one factor for each equivalence class. Hence there are $N^{(\text{number of equivalence classes under $\sim$})} + O(N^{(\text{number of equivalence classes under $\sim$})-1})$ choices for the differences, where the implied constant depends only on $P$ and $\sim$ (and of course on $m$ and the link function). There are $N/k$ choices for the single index since it is already determined mod $k$, therefore the number of cyclic products conforming to $P$ and $\sim$ is at most 
\begin{equation}
    2^m \cdot N^{1+(\text{number of equivalence classes under $\sim$})} + O(N^{(\text{number of equivalence classes under $\sim$})}).
\end{equation} 
Since we have normalized by $\frac{1}{N^{m/2+1}}$, this disappears in the limit unless $m$ is even and all equivalence classes have size $2$. Because for given $k$ there are finitely many patterns and matchings and this number is independent of $N$, the big-$O$ terms and the contributions of patterns and matchings not satisfying the conditions of the lemma contribute $O(1/N)$ to 
$$\frac{1}{N^{m/2+1}}\sum_{1\le i_1,\ldots,i_m\le N} \E \left[a_{i_1,i_2}\cdots a_{i_{m-1},i_{m}}a_{i_m, i_1}\right].$$
This proves that (1) is necessary for the limiting contribution to be nonzero.

(2) is a modification of the proof of \cite[Lemma 2.3]{KKM} for $k$-block circulant ensembles, which is itself modified from \cite{HM,MMS,JMP}; we reprove it here for completeness. By the previous part we may assume $m$ is even, and so write $m=2b$. Writing the indices as $i_{\ell_j}$ for $j \in \{1,\ldots,2b\}$, the matching $\sim$ splits $\{1,\ldots,2b\}$ into $2$-element equivalence classes $\{\ell_1,t_1\},\ldots,\{\ell_b,t_b\}$. Given a matching $\sim$ satisfying (1), the pair of entries corresponding to any pair of matched indices must lie on the same pair of transposes of a given (wrapped) diagonal. Hence the number of cyclic products conforming to $P$ and $\sim$ is bounded above by the number of solutions to the system of congruences 
\begin{equation}\label{eq:eps_congruence}
i_{\ell_{j}+1} - i_{\ell_j} \equiv \curleps_j (i_{t_{j}+1} - i_{t_j}) \pmod{N}, 1 \leq j \leq 2b,
\end{equation}
where the sign $\curleps_j = \pm 1$ is $1$ if the corresponding entries are on the same wrapped diagonal and $-1$ if they are on transposes of the same wrapped diagonal. Now,
\begin{equation}
0 = \sum_{r=1}^{2b} i_{r+1} - i_r = \sum_{j=1}^b (i_{\ell_{j}+1} - i_{\ell_j}) + \sum_{j=1}^b (i_{t_{j}+1} - i_{t_j}) \equiv \sum_{j=1}^m (1+\curleps_j) (i_{t_{j}+1} - i_{t_j}) \pmod{N}.
\end{equation}
If any $\curleps_j = 1$, then the above is a nontrivial congruence among the $i_r$, and consequently a degree of freedom is lost and there is no contribution from the matching $\sim$ in the limit. Hence $\curleps_j = -1$ for all $j$, therefore reducing \eqref{eq:eps_congruence} mod $k$, we have for any cyclic product $a_{i_1,i_2}\cdots a_{i_{m-1},i_{m}}a_{i_m, i_1}$ conforming to $P$ and $\sim$ that
\begin{equation}
r_{\ell_j+1}-r_{\ell_j} \equiv i_{\ell_{j}+1} - i_{\ell_j} \equiv - (i_{t_{j}+1} - i_{t_j}) \equiv r_{t_j+1}-r_{t_j}  \pmod{k}
\end{equation} 
This proves that (2) is necessary for the limiting contribution to be nonzero.

Now, suppose a pattern and matching satisfy (1) and (2). Then choosing the difference $i_{j+1} - i_j$ of two indices mod $N$ specifies the wrapped diagonal the corresponding entry lies on, and hence determines the wrapped diagonal that the matched entry must lie on. Thus there are $(N/k)^{b}$ ways to determine the congruence classes mod $N$ of the pairwise differences of indices, with $N/k$ choices for each one because the congruence class mod $k$ is already specified. Once this is done, there are $N/k$ ways to choose any given index. The normalization in Definition \ref{def:esd} cancels the $N^{1+b}$ factor, leaving $(1/k)^{1+b}$. 
\end{proof}


Since the previous lemma implies that all odd moments are zero, from now on we will consider the $2m \tth$ moment rather than the $m\tth$ moment. In the previous lemma we computed the contribution of a pattern and matching to the limit; now, we compute the contribution of a single pattern, summing over all matchings.

\begin{lemma}\label{lem:pattern_contrib}
Let $(M_{f,N,X})_{N \in k\N}$ be a $k$-periodic family given by link function $f$, and let $P=(i_1,i_2,\ldots,i_{2m})$ be a pattern. 
\begin{equation}\label{eq:pattern_contrib}
   \lim_{N \to \infty} \sum_{\substack{\text{matchings $\sim$}}} \sum_{\substack{\text{Cyclic products $\Pi$ of length $2m$ from $M_{f,N,X}$} \\ \text{conforming to $P$ and $\sim$}}} \E[\Pi] =  (1/k)^{1+m}\E[c_{i_1,i_2}\cdots c_{i_{2m-1},i_{2m}} c_{i_{2m},i_1}],
\end{equation}
where the $c_{i,j}$ are drawn from the complex companion ensemble\footnote{Because this ensemble is $k \times k$, it makes sense to use the $i_j$ as indices, even though these live in $\Z/k\Z$ by Definition \ref{def:pattern} as stated.} $\tilde{M}_{f,k}$. 
\end{lemma}
\begin{proof}

Case $1$: We claim that if the LHS of \eqref{eq:pattern_contrib} is $0$, then the RHS is $0$ as well. If the LHS is zero, then there no matching $\sim$ for which $P$ and $\sim$ satisfy the conditions of Lemma \ref{lem:match_pairs}, by that lemma. Suppose for the sake of contradiction that the RHS $\E[c_{i_1,i_2}\cdots c_{i_{2m-1},i_{2m}} c_{i_{2m},i_1}] \neq 0$. Recall that $c_{i_j,i_{j+1}}$ are either real or complex Gaussians, so in particular are from distributions with odd moments all $0$. The entries which are real Gaussians are either equal to or independent from each other entry, hence for each $t$ such that $c_{i_t,i_{t+1}} \sim \mathcal{N}_\R(0,1)$, there are an even number of $j \in \{1,\ldots,2m\}$ such that $c_{i_j,i_{j+1}} = c_{i_t,i_{t+1}}$ (counting $j=t$). 

The complex Gaussians are independent from all other entries except for those which are equal to their conjugates, hence for each $t$ such that $c_{i_t,i_{t+1}} \sim \mathcal{N}_\C(0,1)$, there are some indices $j \in \{1,\ldots,2m\}$ such that $c_{i_j,i_{j+1}} = c_{i_t,i_{t+1}}$ and some indices $\ell \in \{1,\ldots,2m\}$ such that $c_{i_\ell,i_{\ell+1}} = \overline{c_{i_t,i_{t+1}}}$, and all other random variables $c_{i_r,i_{r+1}}$ are independent of $c_{i_t,i_{t+1}}$. Let $p$ be the number of $j$ (counting $t$) for which $c_{i_j,i_{j+1}} = c_{i_t,i_{t+1}}$, and $q$ be the number of $\ell$ such that $c_{i_\ell,i_{\ell+1}} = \overline{c_{i_t,i_{t+1}}}$. For $a,b$ iid centered Gaussians, $\E[(a+ib)^p(a-ib)^q]$ is nonzero if and only if $p=q$, hence this must be the case.

The above two paragraphs define a partition of $\{1,\ldots,2m\}$ into sets $R_1,\ldots,R_a,C_1,\ldots,C_b$ such that
\begin{itemize}
    \item $c_{i_\ell,i_{\ell+1}}$ is dependent on $c_{i_t,i_{t+1}}$ if and only if $\ell$ and $t$ lie in the same set. 
    \item If $\ell,t \in R_i$ for some $i$, then $c_{i_\ell,i_{\ell+1}}=c_{i_t,i_{t+1}}$, and if $\ell,t \in C_i$ for some $i$, then $c_{i_\ell,i_{\ell+1}}$ and $c_{i_t,i_{t+1}}$ are either equal or complex conjugates (and the conjugates sort in pairs).
    \item The sets $R_i$ and $C_i$ have even cardinality.
\end{itemize}
Since the sets $R_i$ have even cardinality, simply partition each one arbitrarily into pairs $\{t_1,t_2\},\ldots,\{t_{|R_i|-1},t_{|R_i|}\}$ and define $\sim$ on $R_i$ by $t_1 \sim t_2, \ldots, t_{|R_i|-1} \sim t_{|R_i|}$. Partition the sets $C_i$ into pairs $\{t_1,t_2\},\ldots,\{t_{|C_i|-1},t_{|C_i|}\}$ such that $c_{i_{t_j},i_{t_j+1}} = \overline{c_{i_{t_{j+1}},i_{t_{j+1}+1}}}$ for all odd $1 \leq j \leq |C_i|-1$; this is possible because the entries sort into conjugate pairs by the second paragraph. Define $\sim$ similarly as before by making the indices in each pair equivalent.

It is easy to see that $\sim$ satisfies the conditions of Lemma \ref{lem:match_pairs}. Condition (1) is immediate since we defined $\sim$ on disjoint pairs. Suppose $t \sim \ell$ and $c_{i_t,i_{t+1}} = c_{i_\ell,i_{\ell+1}}$ is a real Gaussian. Because the real Gaussian entries on different wrapped diagonals of the complex companion ensemble are independent, and our pair of entries are equal, we know that they come from the same wrapped diagonal, i.e. $i_{t+1}-i_t \equiv i_{\ell+1}-i_\ell \pmod{k}$. Since an entry $c_{i,j}$ of the complex companion ensemble can only be a real Gaussian if $i-j \equiv 0 \text{ or }k/2 \pmod{2}$, hence have $i_{t+1}-i_t \equiv -(i_{\ell+1}-i_\ell) \pmod{k}$, so (2) is satisfied for the indices corresponding to real Gaussian entries. For the indices corresponding to complex Gaussians, we defined $\sim$ to match $\ell \sim t$ for which $c_{i_\ell,i_{\ell+1}} = \overline{c_{i_t,i_{t+1}}}$, hence $i_{\ell+1}-i_\ell \equiv -(i_{t+1}-i_t) \pmod{k}$. This verifies condition (2), hence $P,\sim$ satisfy the conditions of Lemma \ref{lem:match_pairs}. This is a contradiction and thus proves the original claim of Case 1.

Case $2$: Now suppose a nonzero number of matchings satisfy the conditions of Lemma \ref{lem:match_pairs} (along with our fixed pattern $P$). Then for each $1 \leq r \leq \lceil \frac{k}{2} \rceil - 1$, there is a set $S_r \subset \Z$ of values of $j$ for which $i_j-i_{j+1} \equiv r \pmod{k}$, and a corresponding set $\bar{S}_r$ of values of $j$ for which $i_j-i_{j+1} \equiv -r \pmod{k}$, with $|S_r| = |\bar{S}_r| =: n_r$ for some $n_r \geq 0$. Hence for each such $r$ there are $n_r!$ ways to match the indices $j \in S_r$ with exactly one index $j' \in \bar{S}_r$. For $r = 0$ we simply have a set $S_0$ of $2n_0$ values of $j$ for which $i_j - i_{j+1} \equiv 0 \pmod{k}$, so there are $(2n_0-1)!!$ ways to match them. If $k$ is even, then there may be $j \in \{1,\ldots,2m\}$ such that $i_j - i_{j+1} \equiv k/2 \equiv -k/2 \pmod{k}$. Denoting the set of all such $j$ by $S_{k/2}$ and its size by $2n_{k/2}$, there are $(2n_{k/2}-1)!!$ ways to partition $S_{k/2}$ into equivalence classes of size $2$. If $k$ is odd, we adopt the notation that $S_{k/2} = \emptyset$.

Fix such a matching $\sim$ satisfying the conditions of Lemma \ref{lem:match_pairs}; then
\begin{equation*}
    \lim_{N \to \infty} \sum_{\substack{\text{Cyclic products $\Pi$ of length $2m$ from $M_{f,N,X}$} \\ \text{conforming to $P$ and $\sim$}}} \E[\Pi] = (1/k)^{1+m}.
\end{equation*}
Putting this together with the number of matchings satisfying the conditions of Lemma \ref{lem:match_pairs} given $P$, we have that the total contribution of this pattern is 
\begin{equation}
    (1/k)^{1+m}(2n_0-1)!! (2n_{k/2}-1)!! \cdot \prod_{1 \leq r \leq \lceil \frac{k}{2} \rceil - 1} n_r!.
\end{equation}

Each pattern is just a tuple of congruence classes of indices mod $k$, and hence specifies a cyclic product of matrix entries in the finite ensemble $\tilde{M}_{f,k}$. For $a,b$ standard normal, the random variable $a^2+b^2$ has chi-squared distribution with $2$ degrees of freedom, and hence its $m\tth$ moment is $2m!! = 2^m \cdot m!$ (see e.g. \cite{Sim}). The entries $c_{i,j}$ of $\tilde{M}_{f,k}$ which do not lie on the main diagonal (or the split diagonal if $k$ is even) have distribution $\frac{a+bi}{\sqrt{2}}$ for $a,b$ iid standard normal. Hence 
\begin{equation}
    \E[c_{i,j}^m \overline{c_{i,j}}^m] = \frac{2^m \cdot m!}{2^m} = m!.
\end{equation} 
Recall also that the $m\tth$ moment of a standard normal is $(2m-1)!!$. Therefore, using the assumption that entries on distinct diagonals which are not transposes of one another are independent, we have
\begin{align}
(1/k)^{1+m}\E[c_{i_1,i_2}\cdots c_{i_{m-1},i_m} c_{i_m,i_1}] &= (1/k)^{1+m}\E[\prod_{j \in S_0}c_{i_j,i_{j+1}}] \E[\prod_{j \in S_{k/2}}c_{i_j,i_{j+1}}] \prod_{{1 \leq r \leq \lceil \frac{k}{2} \rceil - 1}} \E[\prod_{j \in S_r \cup \bar{S}_r} c_{i_j,i_{j+1}}] \\
&= (1/k)^{1+m}(2n_0-1)!!(2n_{k/2}-1)!! \cdot \prod_{1 \leq r \leq \lceil \frac{k}{2} \rceil - 1} n_r!
\end{align}
where by convention we take $(-1)!! = 1$ in the case where $S_0$ or $S_{k/2}$ is empty. This concludes the proof.
\end{proof}

It follows that the limiting value of the $2m\tth$ moment, which is the total contribution of all patterns, is 
\begin{equation}
\lim_{N \to \infty} \E\left[\nu^{(2m)}_{M_{f,N,X}}\right] = \frac{1}{k^{1+m}}\sum_{1 \leq i_1,\ldots,i_{2m} \leq k} \E[c_{i_1,i_2}\cdots c_{i_{2m},i_1}] = \E[\nu_{\tilde{M}_{f,k}}^{(2m)}]
\end{equation}
where $\nu_A = \frac{1}{k} \sum_{i = 1}^{k} \delta\left(x - \frac{\lambda_{i}}{\sqrt{k}}\right)$ is the empirical spectral measure of a matrix $A$ chosen from $\tilde{M}_{f,k}$. Hence in the limit, the moments of the empirical spectral measure of $M_{f,N,X}$ converge to the moments of the empirical spectral measure of $\tilde{M}_{f,k}$. In the next subsection, we prove that this implies convergence of measures.

\subsection{Almost-sure convergence of empirical spectral measures}

We show the following: If one chooses a matrix $A_N$ randomly from $M_{N \times N}(\R)$ with respect to the law of $M_{f,N,X}$ for $N=k,2k,3k,\ldots$, then with probability $1$ over the product probability space $\prod_{n=1}^\infty M_{N \times N}(\R)$ with product measure (given by the Kolmogorov extension theorem), the sequence of empirical spectral measures $\nu_{A_N}$ converges weakly to $\mu_{\tilde{M}_{f,k}}$ as $N \to \infty$.

\begin{defi}\label{def:weak_conv}
A sequence of random measures $\{\mu_N\}_{N \in \N}$ on a probability space $\Omega$ converges \textbf{weakly almost-surely} to a fixed measure $\mu$ if, with probability $1$ over $\Omega^\N$, we have 
\begin{equation}
\lim_{N \rightarrow \infty} \int f d\mu_N = \int f d\mu
\end{equation}
for all $f \in \mathcal{C}_b(\R)$ (continuous and bounded functions).
\end{defi}

The standard method of moments relies on the following--see for example \cite{Ta}.

\begin{prop}[Moment Convergence Theorem]\label{prop:moment_convergence}
Let $\mu$ be a measure on $\R$ with finite moments $\mu^{(m)}$ for all $m \in \Z_{\geq 0}$, and $\mu_1,\mu_2,\ldots$ a sequence of measures with finite moments $\mu_n^{(m)}$ such that $\lim_{n\rightarrow \infty} \mu_n^{(m)} = \mu^{(m)}$ for all $m \in \Z_{\geq 0}$. If in addition the moments $\mu^{(m)}$ uniquely characterize a measure, then the sequence $\mu_n$ converges weakly to $\mu$.
\end{prop}

The $2m\tth$ moment of the eigenvalue distribution of any complex companion ensemble is 
\begin{equation}\label{eq:momentbound}
    \frac{1}{k^{1+2m}}\sum_{1 \leq i_1,\ldots,i_{2m} \leq k} \E[c_{i_1,i_2} \cdots c_{i_{2m},i_1}] \leq k^{2m}\frac{(2m-1)!!}{k^{1+m}} \leq k^{m-1} (2m-1)!!,
\end{equation}
since there are $k^{2m}$ cyclic products and $\E[c_{i_1,i_2} \cdots c_{i_{2m},i_1}]$ is maximized when all $c_{i_j,i_{j+1}}$ are drawn from the same $\mathcal{N}_\R(0,1)$ distribution on the diagonal. This is bounded above by the $2m\tth$ moment of an $\mathcal{N}(0,k)$ Gaussian, which is $k^{m}(2m-1)!!$. These moments are known classically (and can be easily checked) to satisfy \eqref{eq:carleman} in the following.

\begin{prop}[Carleman's condition] \label{prop:carleman}
Let $\mu$ be a measure with all moments $\mu^{(m)}$ finite for $m \geq 0$. If 
\begin{equation}\label{eq:carleman}
    \sum_{m \geq 1} (\mu^{(2m)})^{-\frac{1}{2m}} = \infty,
\end{equation}
then $\mu$ is the unique measure with moments $\mu^{(m)}$.
\end{prop}

Hence by the bound \eqref{eq:momentbound}, the moments of the eigenvalue distribution of any complex companion ensemble satisfy \eqref{eq:carleman} as well, thus by Proposition \ref{prop:carleman} they determine a unique measure.

Hence we must simply show that for all $m$, $\lim_{N \to \infty} |\nu_{A_N}^{(m)} - \mu_{\tilde{M}_{f,k}}^{(m)}| = 0$ with probability $1$, with the $A_N$ drawn from $M_{f,N,X}$ as described. By the triangle inequality 
\begin{equation}
    |\nu_{A_N}^{(m)} - \mu_{\tilde{M}_{f,k}}^{(m)}| \leq |\nu_{A_N}^{(m)} - \E[\nu_{M_{f,N,X}}^{(m)}]| + | \E[\nu_{M_{f,N,X}}^{(m)}] - \mu_{\tilde{M}_{f,k}}^{(m)}|,
\end{equation}
and the previous section was devoted to showing that $\lim_{N \to \infty} | \E[\nu_{M_{f,N,X}}^{(m)}] - \mu_{\tilde{M}_{f,k}}^{(m)}| = 0$. To show that $\lim_{N \to \infty} |\nu_{A_N}^{(m)} - \E[\nu_{M_{f,N,X}}^{(m)}]| = 0$ with probability $1$, we show that $\E[(\nu_{M_{f,N,X}}^{(m)} - \E[\nu_{M_{f,N,X}}^{(m)}])^4] = O(1/N^2)$, which follows by a slight modification of the arguments used in \cite[Section 6]{HM}. From here the result follows in the standard way from Chebyshev's inequality and the Borel-Cantelli lemma, and the details may again be found in \cite{HM}. This completes the proof of Theorem \ref{thm:main}.

\begin{remark}
If the $k$-periodic family of ensembles is not required to be symmetric, then after removing the $1/N$ normalization in Definition \ref{def:esd} the moments of the empirical spectral measures will converge to the moments of the eigenvalue distribution of a finite $k \times k$ nonsymmetric real ensemble specified by the link function in the same manner, by essentially the same proof but neglecting Lemma \ref{lem:match_pairs}(2), which relies on symmetry. However, as the eigenvalues are no longer real and the moment convergence theorem breaks down for measures on $\C$, this does not allow one to say anything about convergence of the actual empirical spectral measures themselves.
\end{remark}

\begin{remark}
It is natural to ask whether results similar to Theorem \ref{thm:main} can be obtained when the assumption that the matrix entries have finite higher moments is dropped, but this would require a new method of proof.
\end{remark}

\begin{remark}
While Theorem \ref{thm:main} does describe the eigenvalue distribution in terms of a finite complex Hermitian ensemble, at least in the case where this ensemble is the GUE it is possible to further derive a closed form for this distribution using the topological/combinatorial argument of Harer-Zagier in \cite{HarZa}--see also \cite{Led}. It would be interesting to see whether similar arguments work for some other $k$-periodic families.
\end{remark}

\newpage



\end{document}